\documentclass[10pt,a4paper,oneside,reqno]{amsart}

\usepackage{setspace}
\usepackage[textwidth=13cm,textheight=21cm]{geometry}
\usepackage[T1]{fontenc}
\usepackage[dvips]{graphicx}
\usepackage{epsfig}
\usepackage{amsmath,amssymb,amscd,amsthm}
\usepackage{subfigure}
\usepackage[latin1]{inputenc}
\usepackage{latexsym}
\usepackage{graphics}
\usepackage{colortbl} 
\usepackage{color}
\usepackage{ae}
\usepackage{enumitem}
\usepackage[all]{xy}
\usepackage{scalerel}
\usepackage{tikz}
\usepackage{cancel}
\usepackage{bbm,amsbsy}
\usepackage{mathrsfs}  
\usepackage[colorlinks=true,pagebackref=true]{hyperref}
\usepackage[normalem]{ulem}
\usepackage{nicefrac}
\usepackage{xfrac}
\usepackage{faktor}
\usepackage{tikz-cd} 

\usepackage{nomencl}
\usepackage[normalem]{ulem}

\makenomenclature

\makeatletter
\newtheorem*{rep@theorem}{\rep@title}
\newcommand{\newreptheorem}[2]{%
\newenvironment{rep#1}[1]{%
 \def\rep@title{#2 \ref{##1}}%
 \begin{rep@theorem}}%
 {\end{rep@theorem}}}
\makeatother

\makeatletter
\newtheorem*{rep@cor}{\rep@title}
\newcommand{\newrepcor}[2]{%
\newenvironment{rep#1}[1]{%
 \def\rep@title{#2 \ref{##1}}%
 \begin{rep@cor}}%
 {\end{rep@cor}}}
\makeatother

\makeatletter
\newtheorem*{rep@prop}{\rep@title}
\newcommand{\newrepprop}[2]{%
\newenvironment{rep#1}[1]{%
 \def\rep@title{#2 \ref{##1}}%
 \begin{rep@prop}}%
 {\end{rep@prop}}}
\makeatother

\newtheorem{cor}{Corollary}[section]

\newtheorem{theorem}[cor]{Theorem}

\newtheorem{prop}[cor]{Proposition}

\newrepcor{cor}{Corollary}
\newreptheorem{theorem}{Theorem}
\newrepprop{prop}{Proposition}

\newtheorem{lemma}[cor]{Lemma}

\theoremstyle{definition}
\newtheorem{defi}[cor]{Definition}
\theoremstyle{remark}
\newtheorem{remark}[cor]{Remark}
\newtheorem*{remark*}{Remark}

\newtheorem*{notation*}{Notation}

\newlist{steps}{enumerate}{1}
\setlist[steps, 1]{itemsep=8pt,leftmargin=0cm,itemindent=.5cm,labelwidth=\itemindent,labelsep=0cm,align=left,label = \textbf{\emph{Step \arabic*}:\,}}

\newcommand{\C}{{\mathbb C}}
\newcommand{\R}{{\mathbb R}}
\newcommand{\Z}{{\mathbb Z}}
\newcommand{\dev}{\mathrm{dev}}

\newcommand{\E}{\mathbb{E}}

\newcommand{\SO}{\mathrm{SO}}

\newcommand{\grad}{\mathrm{grad}}

\newcommand{\Isom}{\mathrm{Isom}}

\begin{document}

\setcounter{secnumdepth}{2}
\setcounter{tocdepth}{2}

\title[Rigidity of minimal Lagrangian diffeomorphisms]{Rigidity of minimal Lagrangian diffeomorphisms \\ between spherical cone surfaces}

\author[Christian El Emam]{Christian El Emam}
\address{Christian El Emam: Universit\'e du Luxembourg, 
Maison du Nombre,
6 Avenue de la Fonte,
L-4364 Esch-sur-Alzette, Luxembourg.} \email{christian.elemam@uni.lu}

\author[Andrea Seppi]{Andrea Seppi}
\address{Andrea Seppi: Institut Fourier, UMR 5582, Laboratoire de Math\'ematiques,
Universit\'e Grenoble Alpes, CS 40700, 38058 Grenoble cedex 9, France.} \email{andrea.seppi@univ-grenoble-alpes.fr}


\thanks{This work has been partially supported by the LabEx PERSYVAL-Lab (ANR-11-LABX-0025-01) funded by the French program Investissements d'avenir. The first author has been partially
	supported by the FNR OPEN grant CoSH (O20/14766753/CoSH). The authors are members of the national research group GNSAGA}

\maketitle

\vspace{-0.6cm}

\begin{abstract}
We prove that any minimal Lagrangian diffeomorphism between two closed spherical surfaces with cone singularities is an isometry, without any assumption on the multiangles of the two surfaces. As an application, we show that every branched immersion of a closed surface of constant positive Gaussian curvature in Euclidean three-space is a branched covering onto a round sphere, thus generalizing the classical rigidity theorem of Liebmann to branched immersions. 
\end{abstract}

\section{Introduction}

Minimal Lagrangian maps have played an important role in the study of hyperbolic structures on surfaces. As observed independently by Labourie \cite{labourieCP} and Schoen \cite{Schoenharmonic}, given two closed hyperbolic surfaces $(\Sigma_1,h_1)$ and $(\Sigma_2,h_2)$, there exists a unique minimal Lagrangian diffeomorphism in the homotopy class of every diffeomorphism $\Sigma_1\to\Sigma_2$.  See also \cite{zbMATH00827487} and \cite{zbMATH00870993,graham} for extensions of this result. Alternative proofs have been provided later, in the context of Anti-de Sitter three-dimensional geometry (see \cite{foliationCMC} and \cite[\S 7]{survey}), using higher codimension mean curvature flow (see \cite{zbMATH01744223} and \cite{zbMATH05996706}). Using Anti-de Sitter geometry, the result of Labourie and Schoen has been generalized in various directions: in \cite{bon_schl,zbMATH06902491} in the setting of universal Teichm\"uller space; in \cite{zbMATH06680278} for closed hyperbolic surfaces with cone singularities of angles in $(0,\pi)$, provided the diffeomorphism $\Sigma_1\to\Sigma_2$ maps cone points to cone points of the same angles. Toulisse then proved in \cite{zbMATH07149810} the existence of minimal maps between closed hyperbolic surfaces of different cone angles, by purely analytic methods. We remark that interesting results in a similar spirit have been obtained for minimal Lagrangian diffeomorphisms between bounded domains in the Euclidean plane (\cite{zbMATH00032290,zbMATH01145655}) and in a complete non-positively curved Riemannian surface (\cite{zbMATH05346543}).

On the other hand, spherical metrics with cone singularities on a closed surfaces have been studied in \cite{zbMATH03988171,zbMATH04073390,zbMATH04142953,zbMATH04194621,zbMATH00126050}. Very recently the works \cite{zbMATH06992339,zbMATH07091755,mondpanov2}, by geometric methods, and  \cite{zbMATH06741332,zbMATH07197533,mazzeozhu2}, by analytic methods, developed the study of the deformations spaces of spherical cone metric, highlighting their complexity. 

\subsection{Main statement}

It thus seems a natural question to ask whether one can find a minimal Lagrangian diffeomorphism between two spherical cone surfaces.  In this paper we answer negatively to this question, without any assumption on the cone angles. We show that two spherical cone surfaces do not admit any minimal Lagrangian diffeomorphism unless they are isometric. When they are isometric, the only minimal Lagrangian diffeomorphisms are isometries. We summarize these statements as follows:

\begin{theorem}\label{thm:main1}
Given two closed spherical cone surfaces $(\Sigma_1,\mathfrak p_1,g_1)$ and $(\Sigma_2,\mathfrak p_2,g_2)$, any minimal Lagrangian diffeomorphism  $\varphi:(\Sigma_1,\mathfrak p_1,g_1)\to(\Sigma_2,\mathfrak p_2,g_2)$ is an isometry.
\end{theorem}

We remark that, as part of our definition (Definition \ref{defi min lag}), a minimal Lagrangian diffeomorphism $\varphi$ is a smooth diffeomorphism between $\Sigma_1\setminus\mathfrak p_1$ and $\Sigma_2\setminus\mathfrak p_2$ that extends \emph{continuously} to the cone points. \emph{A priori}, we do not assume that such a smooth map extends \emph{smoothly} to the cone points. This subtlety is at the origin of an important technical point in our proof, which is summarized in Section \ref{sec:outline} below. 
 
\subsection{Surfaces of constant Gauss curvature} 

We provide an application of our main result for branched immersions of surfaces of constant Gaussian curvature in Euclidean three-space, generalizing the classical Liebmann's theorem which states that every closed immersed surface of positive constant Gaussian curvature in Euclidean space is a round sphere. 

In \cite{zbMATH06243957}, G\'alvez, Hauswirth and Mira classified the \emph{isolated singularities} of surfaces of constant Gaussian curvature. According to their definition, isolated singularities of an immersion $\sigma:U\setminus\{p\}\to\R^3$, for $U$ a disc, are those that extend continuously on $U$.  Among these, they considered \emph{extendable singularities}, namely those for which the normal vector extends smoothly at $p$, and showed that they are either \emph{removable}, meaning that they extend to an immersion of $U$, or \emph{branch points}, meaning that the Gauss map is locally expressed as the map $z\mapsto z^k$ with respect to some coordinates on $U$ and on $\mathbb S^2$.  {In the following, we will use the term \emph{branched immersion} of a surface $\Sigma$ to indicate an immersion of the complement of a discrete set $D$ {of isolated singularities} of $\Sigma$, that extends continuously to $D$ and has a branch point at every point of $D$, according to the above definition.} Here we show a rigidity result for branched immersions of closed surfaces:

\begin{cor}\label{cor:rigidity}
Every branched immersion in Euclidean three-space of a closed surface of constant positive Gaussian curvature is a branched covering onto a round sphere.
\end{cor}

As we mentioned, Corollary \ref{cor:rigidity} can be regarded as a generalization of Liebmann's theorem, which we indeed recover by an independent proof when the immersion has no branch points. Roughly speaking, we prove Corollary \ref{cor:rigidity} by applying Theorem \ref{thm:main1} to the Gauss map of a branched immersion $\sigma:\Sigma\to\R^3$, which induces a minimal Lagrangian self-diffeomorphism of $\Sigma$ with respect to the first and third fundamental form, both of which are spherical cone metrics. 

Finally, we remark that the hypothesis that the surface $\Sigma$ is closed is essential in Corollary \ref{cor:rigidity}, as well as the closedness of $\Sigma_1$ and $\Sigma_2$ in Theorem \ref{thm:main1}. Indeed one can find \emph{local} deformations of spheres of constant Gaussian curvature, with branch points (see \cite{zbMATH06599450} for many examples) or without (for instance by surfaces of revolution); their Gauss maps provide non-isometric minimal Lagrangian diffeomorphisms between open spherical surfaces (with or without cone points).

\subsection{Outline of the proof of Theorem \ref{thm:main1}.} \label{sec:outline}
 
A map $\varphi:(\Sigma_1,\mathfrak p_1,g_1)\to(\Sigma_2,\mathfrak p_2,g_2)$ is minimal Lagrangian if it is area-preserving and its graph (restricted to the nonsingular locus) is minimal in the product $\Sigma_1\times\Sigma_2$. A useful characterization  is that one can express (on the nonsingular locus) $\varphi^*g_2=g_1(b\cdot,b\cdot)$ for $b$ a (1,1) tensor which is self-adjoint with respect to $g_1$, positive definite, and satisfies the conditions $d^{\nabla^{g_1}}\!b=0$ and $\det b=1$. For the sake of completeness, we prove the equivalence of the two definitions in Appendix \ref{app:defis}. From this characterization, one sees that minimal Lagrangian maps are those that can be \emph{locally} realized as the Gauss maps of surfaces of constant Gaussian curvature one in Euclidean three-space, as a consequence of the Gauss-Codazzi equations.

Starting by this characterization, using the spherical metric $g_1$ and the (1,1) tensor $b$, we produce a pair $(G,B)$, {defined on the complement of the cone points of $\Sigma_1$} where $G$ is a Riemannian metric on $\Sigma_1$, and $B$ a (1,1) $G$-self-adjoint traceless tensor, satisfying the equations
$$d^{\nabla^G}\!\!B=0\qquad\text{and}\qquad K_G=1+\det B~.$$
Although we will not use spherical three-dimensional geometry in this paper, we remark that these are precisely the Gauss-Codazzi equations for a surface in $\mathbb S^3$, which is minimal since $B$ is traceless. Equivalently, by the Lawson correspondence, the pair $(G,\mathbbm 1+B)$ satisfies the Gauss-Codazzi equations for a surface of constant mean curvature one in $\R^3$. Such constant mean curvature surface is realized (at least locally) as the parallel surface from the surface of constant Gaussian curvature mentioned above, which is determined by the pair $(g_1,b)$.
Assuming $\varphi:(\Sigma_1,\mathfrak p_1,g_1)\to(\Sigma_2,\mathfrak p_2,g_2)$ is a minimal Lagrangian diffeomorphism, the goal of the proof is to show that $B$ vanishes identically, which is equivalent to $\varphi$ being an isometry.

For this purpose, assuming by contradiction that $B$ does not vanish identically, the next step consists in computing the Laplace-Beltrami operator of the function $\chi$ defined, in the complement of the zeros of $B$, as the logarithm of the positive eigenvalue of $B$ (up to a certain constant). It turns out that $\Delta^G\chi$ equals the curvature of the metric $G$, which is positive (Corollary \ref{cor subharmonic}), hence $\chi$ is subharmonic and negative and the contradiction is then obtained by an application of the maximum principle. 

However, it is essential to prove that the metric $G$ has the conformal type of a punctured disc in a neighbourhood of every cone point of $\Sigma_1$. This would be automatically satisfied assuming some additional regularity on the minimal Lagrangian map $\varphi$: for instance, if $\varphi$ is supposed quasiconformal, which is equivalent to boundedness of the (1,1) tensor $b$, then $g_1$ and $G$ are quasiconformal, and therefore both $g_1$ and $G$ have the conformal type of a punctured disc near the cone points. 
But, as we mentioned above, in our Theorem \ref{thm:main1} we assume a weaker regularity on $\varphi$ at the cone points, namely we only suppose that $\varphi$ is continuous. 

To prove that $G$ has the conformal type of a punctured disc around the cone points, we apply the interpretation in terms of surfaces in Euclidean space, and we show that $G$ can be realized in a punctured neighbourhood $U^*$ of any cone point as the metric induced by the first fundamental form of an equivariant immersion of $\widetilde{U^*}$ in $\R^3$. We also prove that the normal vector of the equivariant immersion admits a limit, and the vertical projection induced a bi-Lipschitz equivalence between $G$ and a flat metric on $U^*$. A complex analytic argument, based on Schwarz Reflection Principle, shows that this flat metric has the conformal type of $\mathbb D^*$ at the puncture, and this implies that $G$ has the conformal type of $\mathbb D^*$ as well.

\subsection*{Acknowledgements}
We would like to thank the anonymous referees for several useful comments, and in particular for the comment that led to Remark \ref{referee}.


\tableofcontents

\section{Definitions and setup}

Let us start by introducing the fundamental definitions and some well-known properties.

\subsection{Conical metrics}
 We give the general definition of cone Riemannian metric.
 
\begin{defi} \label{defi cone sing}
Given a smooth surface $\Sigma$ and a discrete subset $\mathfrak p\subset \Sigma$, a \emph{cone} metric on $\Sigma$ is a Riemannian metric $g$ on $\Sigma\setminus \mathfrak p$ such that can be written in a punctured neighbourhood $U\setminus\{p\}$ of every point $p\in \mathfrak p$ as: 
\begin{equation} \label{eq: conformal model}
g=e^{2f}|z|^{2\alpha-2}|dz|^2
\end{equation}
with respect to a coordinate $z:U\to\C$, for $f:U\setminus \{p\}\to\R$ a smooth bounded function and $\alpha\in (0,1)\cup (1,+\infty)$.  
\end{defi}

The subset $\mathfrak p$ is called the \emph{singular locus}, and its complement the \emph{regular locus}. It will be convenient to set $\theta:=2\pi\alpha$, which is called the \emph{cone angle} at $p$.

A cone metric is called  \emph{spherical} when it has constant curvature $+1$ on the regular locus. In this cases one has the following explicit local expression for the metric tensor \eqref{eq: conformal model}:

\begin{equation}\label{eq:spherical cone metric} 
g=\frac{4\alpha^2|z|^{2\alpha-2}}{(1+|z|^{2\alpha})^2}|dz|^2~.
\end{equation}



\subsection{Minimal Lagrangian maps}

Let us now move on to the definition of minimal Lagrangian maps.

\begin{defi} \label{defi min lag}
Given two spherical cone surfaces $(\Sigma_1,\mathfrak p_1,g_1)$ and $(\Sigma_2,\mathfrak p_2,g_2)$, a \emph{minimal Lagrangian diffeomorphism} is a diffeomorphism $\varphi:\Sigma_1\setminus\mathfrak p_1\to\Sigma_2\setminus\mathfrak p_2$, that extends to a homeomorphism between $\Sigma_1$ and $\Sigma_2$, having the property that the unique $g_1$-self-adjoint, positive definite $(1,1)$ tensor $b$ on $\Sigma_1\setminus \mathfrak p_1$ such that $\varphi^*g_2=g_1(b\cdot,b\cdot)$ satisfies the conditions: 
\begin{equation}\label{eq:defi min lag}
\det b=1\qquad\text{and}\qquad d^{\nabla^{g_1}}\!b=0~.
\end{equation}
\end{defi}

Here and in what follows, $\nabla^g$ denotes the Levi-Civita connection of a Riemannian metric $g$. We recall that, for a connection $\nabla$ and a $(1,1)$ tensor $A$, the \emph{exterior derivative} $d^\nabla \!\!A$ is defined as 
$$d^\nabla \!\!A(v,w)=\nabla_v(A(w))-\nabla_w(A(v))-A([v,w])~.$$
A tensor satisfying $d^{\nabla^g}\!\!A=0$ is called \emph{Codazzi tensor} with respect to the metric $g$.

For the sake of completeness, in Appendix \ref{app:defis} we shall prove that Definition \ref{defi min lag} is equivalent to the fact that the graph of $\varphi$ is minimal Lagrangian in $\Sigma_1\times\Sigma_2$, thus justifying the terminology.

\begin{remark}
It is natural to require that a minimal Lagrangian map maps cone points to cone points, as in Definition \ref{defi min lag}. Indeed (as we explain in Remark \ref{rmk extend}), if $\varphi:(U_1,p_1,g_1)\to (U_2,p_2,g_2)$ is minimal Lagrangian diffeomorphism between two punctured discs endowed with metrics of the form \eqref{eq:spherical cone metric}, then the cone angles of $g_1$ and $g_2$ are necessarily equal. In particular, if the ``cone angle'' is $2\pi$ for $g_1$ at $p_1$, meaning that the metric extends to a smooth spherical metric on the disc, then the same holds for $g_2$ at $p_2$, and moreover in this case $\varphi$ extends smoothly to a minimal Lagrangian diffeomorphism between $U_1$ and $U_2$.
\end{remark}

When a metric is written in the expression $g(A\cdot,A\cdot)$ for $A$ an invertible Codazzi tensor with respect to $g$, its connection and curvature are easily related to those of $g$, as in the following well-known lemma.

\begin{lemma}[{\cite[Proposition 3.12]{zbMATH05200423}}] \label{lemma KS}
Let $g$ be a Riemannian metric on a surface $\Sigma$ and let $A$ a smooth $(1,1)$ tensor with $d^{\nabla^g}\!\!A=0$ such that $\det A$ vanishes nowhere. Define $h=g(A\cdot,A\cdot)$. Then the Levi-Civita connections of $g$ and $h$ are related by:
\begin{equation}\label{eq:connection}
\nabla^h_v w=A^{-1}\nabla^g_v(A(w))~,
\end{equation}
and their curvatures by:
\begin{equation}\label{eq:curvature}
K_h=\frac{K_g}{\det A}~.
\end{equation}

\end{lemma}

\begin{remark} \label{rmk:detb}
This lemma has two immediate consequences. First, it turns out that the condition $\det b=1$ is actually redundant in Definition \ref{defi min lag}. Indeed, assuming $d^{\nabla_{g_1}}b=0$, it follows from $K_{g_1}=K_{\varphi^*g_2}=1$ and from Equation \eqref{eq:curvature}  that
$\det b=1$. Second, the inverse of a minimal Lagrangian diffeomorphism is minimal Lagrangian, since one can write $g_1=\varphi^*g_2(b^{-1}\cdot,b^{-1}\cdot)$ and it is easily checked that $b^{-1}$ is self-adjoint and Codazzi for $\varphi^*g_2$, using \eqref{eq:connection}. Hence $\varphi_*b^{-1}:=(d\varphi)\circ b^{-1}\circ (d\varphi)^{-1}$ satisfies the conditions in the Definition \ref{defi min lag} for $\varphi^{-1}:\Sigma_2\to\Sigma_1$.
\end{remark}

\subsection{Defining the pair $(G,B)$}
We now introduce the fundamental construction for our proofs.

\begin{defi} \label{defi GB}
Given a minimal Lagrangian map $\varphi:(\Sigma_1,\mathfrak p_1,g_1)\to (\Sigma_2,\mathfrak p_2,g_2)$, we define on $\Sigma_1\setminus\mathfrak p_1$ a Riemannian metric 
$$G=\frac{1}{4}g_1((\mathbbm 1+b)\cdot,(\mathbbm 1+b)\cdot)~,$$
and a $(1,1)$-tensor
$$B=(\mathbbm 1+b)^{-1}(b-\mathbbm 1)~,$$
for $b$ as in Definition \ref{defi min lag}.
\end{defi}


\begin{remark}\label{rmk inverse 2}
Definition \ref{defi GB} has a symmetry with respect to $g_1$ and $g_2$. More precisely, the metric $G'$ on $\Sigma_2\setminus \mathfrak p_2$ associated to the map $\varphi^{-1}:(\Sigma_2,\mathfrak p_2,g_2)\to (\Sigma_1,\mathfrak p_1,g_1)$, which is again minimal Lagrangian (Remark \ref{rmk:detb}), is isometric to the metric  $G$ on $\Sigma_1\setminus \mathfrak p_1$. 

To see this, we have observed in Remark \ref{rmk:detb} that the (1,1) tensor associated to the minimal Lagrangian map $\varphi^{-1}$ in Definition \ref{defi min lag} is $\varphi_*b^{-1}$. Hence $G'$ is the metric on $\Sigma_2$ defined by $G'=(1/4)g_2((\mathbbm 1+\varphi_*b^{-1})\cdot,(\mathbbm 1+\varphi_*b^{-1})\cdot)$, and one sees immediately that $\varphi^*G'=G$. Similarly, one finds $B'=-\varphi_*B$. 
\end{remark}

It is immediate to check that $B$ is $G$-self-adjoint, since $b$ is $g_1$-self-adjoint. The following lemma is an immediate algebraic computation.

\begin{lemma} \label{prop eigenvalues}
The eigenspaces of $B$ coincide with those of $b$, and if we denote by $\lambda$ and $1/\lambda$ the eigenvalues of $b$, then the eigenvalues of $B$ are
$$\Lambda=\frac{\lambda-1}{1+\lambda}\qquad\text{and}\qquad \Lambda'=\frac{1-\lambda}{1+\lambda}=-\Lambda~.$$
In particular, $B$ is traceless and $|\Lambda|,|\Lambda'|<1$. Finally, at any point  we have $B=0$ if and only if $b=\mathbbm 1$.  
\end{lemma}

We observe that the pair $(G,B)$ satisfies the following important properties:

\begin{prop}\label{prop GCS3}
The following Codazzi equations are satisfied by the pair $(G,B)$:
$$d^{\nabla^G}\!\!B=0\qquad~.$$
Moreover, the curvature of $G$ is positive.
\end{prop}
\begin{proof}
This is a straightforward verification using Lemma \ref{lemma KS}. Indeed, by Equation \eqref{eq:connection} 
$$d^{\nabla^G}\!\!B=(\mathbbm 1+b)^{-1}d^{\nabla^{g_1}}\!(b-\mathbbm 1)=0$$
since both $\mathbbm 1$ and $b$ are $g_1$-Codazzi. For the curvature condition, using Equation \eqref{eq:curvature} we have
$$K_G=\frac{4K_{g_1}}{\det(\mathbbm 1+b)}=\frac{4}{2+\mathrm{tr}b}>0~.$$
This concludes the proof.
\end{proof}

\begin{remark} \label{rmk lawson}
A simple computation shows moreover that 
$$1+\det B=1+\frac{2-\mathrm{tr}b}{2+\mathrm{tr}b}=\frac{4}{2+\mathrm{tr}b}=K_G~.$$
In other words, together with Proposition \ref{prop GCS3}, we see that the pair $(G,B)$ satisfies the Gauss-Codazzi equation when the ambient manifold is $\mathbb S^3$. However, in this paper we will not use spherical geometry; the construction of Section \ref{sec:surfaces} is motivated by the observation that $(G,\mathbbm 1+B)$ satisfy the Gauss-Codazzi equations in Euclidean space $\E^3$. 

This is the so-called \emph{Lawson correspondence} introduced in \cite{zbMATH03327055}: as a consequence of $\mathrm{tr} B=0$, the Gauss equation $K_G=1+\det B$ is equivalent to $K_G=\det(\mathbbm 1+B)$, namely the Gauss equation in $\E^3$; furthermore $\mathbbm 1+B$ is clearly Codazzi with respect to $G$. In summary, when lifted to the universal cover of $\Sigma_1\setminus\mathfrak p_1$, $(G,B)$ provide the immersion data of a minimal surface in $\mathbb S^3$, while $(G,\mathbbm 1+B)$ those of a constant mean curvature one surface in $\mathbb E^3$.  
\end{remark}

\section{A maximum principle}

The key idea in the proof of Theorem \ref{thm:main1} is an application of the maximum principle to show that $B$ is identically zero, that is $b$ is the identity operator (by Lemma \ref{prop eigenvalues}). This will show that any minimal Lagrangian map $\varphi:(\Sigma_1,\mathfrak p_1,g_1)\to (\Sigma_2,\mathfrak p_2,g_2)$ is an isometry.

\subsection{A bounded subharmonic function}

The fundamental relation involved in our application of the maximum principle is a consequence of the following formula, presented in \cite[Lemma 3.11]{zbMATH05200423}. Since this is a fundamental step, we provide a quick proof for convenience of the reader. In Appendix \ref{app:alternative} we give another short proof, entirely based on the fact that $G(B\cdot,\cdot)$ is the real part of a holomorphic quadratic differential. 

\begin{lemma}\label{lemma laplace}
Let $G$ be a Riemannian metric on a surface $U$ and $B$ a traceless $G$-self-adjoint, $G$-Codazzi $(1,1)$ tensor that does not vanish on $U$. Denote $\chi=(1/4)\log(-\det B)$. Then
$$K_G=\Delta^G\chi~.$$
\end{lemma}

Here we denote by $\Delta^G$ the Laplace-Beltrami operator of $G$, with negative spectrum. 

\begin{proof}
Let $e,e'$ be an oriented orthonormal frame of eigenvectors of $B$, so that $B(e)=\Lambda e$ and $B(e')=\Lambda' e'=-\Lambda e'$. Let us denote by $\omega$ the  connection form associated to the Levi-Civita connection $\nabla^G$ of $G$, which satisfies 
\begin{equation}\label{referee2}
\nabla_v e=\omega(v)e'\quad\text{ and }\quad\nabla_v e'=-\omega(v)e~,
\end{equation} 
where to simplify the notation we set $\nabla=\nabla^G$. Since $B$ does not vanish on $U$ by hypothesis, we can assume moreover that $\Lambda$ is the positive eigenvalue of $B$, so that $\chi=(1/4)\log(\Lambda^2)=(1/2)
\log\Lambda$. 

First, let us compute the Codazzi condition applied to the frame $\{e,e'\}$:
\begin{align*}
0&=\nabla_eB(e')-\nabla_{e'}B(e)-B(\nabla_e e'-\nabla_{e'} e)\\
&=-(\partial_e\Lambda) e'-(\partial_{e'}\Lambda) e-2\Lambda\nabla_e e'-2\Lambda\nabla_{e'} e~,
\end{align*}
where we have used Equation \eqref{referee2} from the first to the second line.
Hence we get $\partial_e\Lambda=-2\Lambda\omega(e')$ and $\partial_{e'}\Lambda=2\Lambda\omega(e)$. In terms of $\chi=(1/2)\log \Lambda$, we have
$\partial_e\chi=-\omega(e')$ and $\partial_{e'}\chi=\omega(e)$. 

Second, we compute
\begin{align*}
K_G&=-d\omega(e,e')=-\partial_e\omega(e')+\partial_{e'}\omega(e)+\omega(\nabla_e e'-\nabla_{e'} e)\\
&=\partial_e\partial_e\chi+\partial_{e'}\partial_{e'}\chi-\omega((\partial_e \chi) e')
-\omega((\partial_{e'}\chi) e)\\
&=\mathrm{Hess}\chi(e,e)+\mathrm{Hess}\chi(e',e')=\Delta^G\chi~,
\end{align*}
where from the second to the third line we used 
\begin{align*}
\mathrm{Hess}\chi(e,e)&=\partial_e\partial_e\chi-\partial_{\nabla_ee}\chi =\partial_e\partial_e\chi-\partial_{\omega(e)e'}\chi \\
&=\partial_e\partial_e\chi-\omega((\partial_{e'}\chi)e)=\partial_e\partial_e\chi-\omega(\omega(e)e) \\
&=\partial_e\partial_e\chi+\omega(\nabla_ee')~,
\end{align*}
and similarly for $\mathrm{Hess}\chi(e',e')$.
\end{proof}

Proposition \ref{prop GCS3} and Lemma \ref{lemma laplace} show that, on the subset of $\Sigma_1\setminus\mathfrak p_1$ where $B\neq 0$, 
$$\Delta^G\chi> 0~.$$
We have thus shown that $\chi$ is subharmonic. Moreover $\chi$ is negative, because  $|\Lambda|<1$ by Lemma \ref{prop eigenvalues}. We summarize these facts as follows:

\begin{cor}\label{cor subharmonic}
The function $\chi=(1/2)\log|\Lambda|$ is negative and satisfies $\Delta^G\chi>0$ on the complement of the zeros of $B$ in $\Sigma_1\setminus\mathfrak p_1$.
\end{cor}

\begin{remark}
Although not essential in the proof, we remark that given a smooth $(1,1)$ tensor $A$, $A$ is $g$-self-adjoint and traceless if and only if $g(A\cdot,\cdot)$ is the real part of a holomorphic quadratic differential (\cite{zbMATH03064601}, see Proposition \ref{prop hopf}). Hence either $B\equiv 0$ or $B$ vanishes on a discrete subset of $\Sigma_1\setminus\mathfrak p_1$.
\end{remark}

\subsection{Proof of Theorem \ref{thm:main1}}

The main idea of the proof is to apply a maximum principle argument to the function $\chi$ of the previous section. To control the behaviour of $\chi$ at the singularities, we need the following statement on the conformal type of the metric $G$. Its proof is postponed to Section \ref{sec:surfaces}.

\begin{prop}\label{prop:conf type}\label{lemma realize min lag}
Let $U_i$ be a disc endowed with a spherical metric $g_i$ with a cone point at $p_i$, for $i=1,2$, let $\varphi:(U_1,p_1,g_1)\to (U_2,p_2,g_2)$ be a minimal Lagrangian diffeomorphism, and let $b$ the $(1,1)$ tensor as in Definition \ref{defi min lag}.
Then the conformal structure induced by the metric $G=(1/4)g_1((\mathbbm 1+b)\cdot,(\mathbbm 1+b)\cdot)$ on a neighbourhood of $p_1$ is biholomorphic to $\mathbb D^*$.
\end{prop}

Assuming Proposition \ref{prop:conf type}, we now conclude the proof of Theorem \ref{thm:main1}. 

\begin{proof}[Proof of Theorem \ref{thm:main1}]
Assume by contradiction that $B$ does not vanish identically. Recall that the function $\chi=(1/4)\log(-\det B)$ is negative and subharmonic by Corollary \ref{cor subharmonic}. We can extend $\chi$ to a function on $\Sigma_1\setminus \mathfrak p_1$, with values in $[-\infty,+\infty)$. 

Now, pick a cone point $p\in\mathfrak p_1$ and a neighbourhood $U$ of $p$ in which the metric $G$ is biholomorphic to $\mathbb D^*$. We claim that 
\begin{equation}\label{eq:max on bdy}
\sup_U \chi=\max_{\partial U}\chi~.
\end{equation}
This will conclude the proof, since it implies that $\chi$ has a maximum point in $\Sigma_1\setminus \mathfrak p_1$, and this contradicts Corollary \ref{cor subharmonic}.

Let us prove \eqref{eq:max on bdy}. Let us pick a biholomorphic chart $z:U\to\mathbb D^*$, and consider $\chi$ as a function of $z$. The metric $G$ is {smooth in the complement of $p$, and is expressed in the $z$-coordinate on $\mathbb D^*$ as $e^{2f}|dz|^2$.} {(We remark that we are not making any assumption on the behaviour of $f$ close to $0$. We will see in Remark \ref{rmk same cone angles} that $G$ has a cone point in $p$, which gives information on the behaviour of $f$ as in Definition \ref{defi cone sing}. However, we do not need to use this fact here.)} 

Hence the Laplace-Beltrami operator $\Delta^G$ equals $e^{-2f}\Delta$, where $\Delta$ is the flat Laplacian on the disc. By Corollary \ref{cor subharmonic}, this implies $\Delta\chi>0$ on $\mathbb D^*$. Now choose any $\epsilon>0$. The function $\chi_\epsilon(z):=\chi(z)+\epsilon\log|z|$ still satisfies $\Delta\chi_\epsilon>0$ because $\log|z|$ is harmonic, and coincides with $\chi$ on $\mathbb S^1$. Moreover $\chi_\epsilon$ tends to $-\infty$ at $0$, because $\chi$ is bounded above. By the maximum principle, $\chi_\epsilon$ cannot have an interior maximum point, hence 
$$\chi_\epsilon(z)\leq \max_{\mathbb S^1}\chi_\epsilon=\max_{\mathbb S^1}\chi$$
for any $z\in \mathbb D^*$. It follows that 
$$\chi(z)\leq \max_{\mathbb S^1}\chi-\epsilon\log|z|~.$$
Since $\epsilon$ was chosen arbitrarily, this shows \eqref{eq:max on bdy} and concludes the proof.
\end{proof}

\begin{remark}\label{referee}
An anonymous referee remarked that, in order to prove \eqref{eq:max on bdy}, one can apply the result that a bounded subharmonic function $\chi$ on $\mathbb D^*$ extends to a subharmonic function on $\mathbb D$ (see \cite[Theorem 5.18]{zbMATH03652704}), and then conclude by the maximum principle applied to the extension of $\chi$.
We have preferred to stick to the  more elementary argument  for \eqref{eq:max on bdy} in the proof of Theorem \ref{thm:main1}.
\end{remark}

\section{Immersions in Euclidean space}\label{sec:surfaces}


In order to complete the proof of Theorem \ref{thm:main1} it only remains to prove Proposition \ref{prop:conf type}. We then prove Corollary \ref{cor:rigidity}. The relation between minimal Lagrangian diffeomorphisms and immersions in Euclidean space will play an essential role for both results.

\subsection{Proof of Proposition \ref{prop:conf type}}

The guiding idea towards Proposition \ref{lemma realize min lag} is that, given the tensor $b$ as in Definition \ref{defi min lag}, the pair $(g_1,b)$ represents locally the embedding data of an immersed surface of constant Gaussian curvature one in Euclidean space, by the fundamental theorem of surfaces; moreover, there is a parallel constant mean curvature surface whose first fundamental form is $G$ up to a factor (and whose shape operator is $\mathbbm 1+B$, compare Remark \ref{rmk lawson}). However, since $\Sigma_1$ is not simply connected, we will need to lift $(g_1,b)$ to its universal cover, and refine this approach in order to deal with the equivariance of the obtained immersion. Moreover, in the proof we find convenient to switch the roles of $g_1$ and $g_2$, namely we apply the above guiding idea to $\varphi^{-1}$, see Remark \ref{rmk inverse 2}.

We will apply the following result.

\begin{lemma}[{\cite{zbMATH03682587},\cite[Proposition 1.3.3]{zbMATH03791189}}] \label{lemma codazzi}
Given a simply connected Riemannian manifold $(M,g)$ of constant sectional curvature $K$ and a self-adjoint $(1,1)$ tensor $A$ satisfying the Codazzi equation $d^{\nabla^g}\!A=0$, there exists a smooth function $u:M\to\R$ such that 
$$A=\nabla^g_\bullet\grad^gu+Ku\mathbbm 1~.$$
\end{lemma}

We observe that the term $\nabla^g_\bullet\grad^gu$ is the Hessian of $u$ as a $(1,1)$ tensor, i.e. $\mathrm{Hess}^g u(v,w)=g(\nabla^g_v\grad^gu,w)$.

\begin{proof}[Proof of Proposition \ref{lemma realize min lag}]
To simplify the notation, let us denote $U_i^*:=U_i\setminus\{p_i\}$. 
Lift $g_1$ and $b$ to the universal cover $\widetilde{U_1^*}$. Then $\widetilde b$ still satisfies the Codazzi equation with respect to $\widetilde g_1$. By Lemma \ref{lemma codazzi} there exists a function $u:\widetilde {U_1^*}\to\R$ such that
$$\widetilde b=\nabla^{\widetilde g_1}_\bullet\grad^{\widetilde g_1} u+u\mathbbm 1~.$$
Pick also a developing map $\dev:\widetilde {U_1^*}\to\mathbb S^2\subset\E^3$ for the spherical structure of $g_1$ on $U_1^*$, namely $\dev$ is a local isometry with respect to the metric $\widetilde g_1$ on $\widetilde {U_1^*}$. We define the maps $\sigma,\varsigma:\widetilde {U_1^*}\to\E^3$ by
$$\sigma(x)=\dev_*(\grad^{\widetilde g_1}u)+u\dev(x) \qquad \varsigma(x)=\frac{1}{2}\left(\dev_*(\grad^{\widetilde g_1}u)+(u+1)\dev(x)\right)~.$$
Observe that $\varsigma=(1/2)(\sigma+\dev)$.

\begin{steps}
\item Let us show  that the first fundamental forms of $\sigma$ and $ \varsigma$  are the lifts to the universal cover of $\varphi^*g_2$ and $G$, respectively. 
Being a local statement, we can isometrically identify an open neighbourhood of any point of $\widetilde {U_1^*}$ with a subset of $\mathbb S^2$, so that $\dev$ is the identity. By differentiating $\sigma$ with this identification, we see that
\begin{align*}
D_v \sigma_x&=D_v\grad^{\mathbb S^2}\!\! u+\langle \grad^{\mathbb S^2}\!\!u,v\rangle x+uv \\
&=\nabla^{\mathbb S^2}_v\!\grad^{\mathbb S^2}\!\! u+uv=\widetilde b(v)~,
\end{align*}
where we used $D$ to denote the ambient derivative in $\E^3$, $\langle\cdot,\cdot\rangle$ the metric of $\E^3$ (which restricts to the metric of $\mathbb S^2$), and from the first to the second line we applied the fact that the second fundamental form of $\mathbb S^2$ equals $-\langle\cdot,\cdot\rangle$ with respect to the outer unit normal. This computation has several consequences, namely:
\begin{enumerate}
\item The map $ \sigma$ is an immersion, since its differential is nonsingular. 
\item The normal vector of $ \sigma$ at a point $x\in \widetilde{U_1^*}$ is $\dev(x)$. Indeed, in the above identification, $x$ itself is orthogonal to the image of the differential of $\sigma$ at $x$. In other words, the Gauss map of $ \sigma$ is $\dev$.
\item The first fundamental form of $ \sigma$ equals
$\langle D_v\sigma_x,D_w\sigma_x\rangle=\widetilde g_1( \widetilde b(v),\widetilde b(w))=\widetilde{\varphi^*g_2}(v,w)$.
\item The shape operator of $ \sigma$ is $\widetilde b^{-1}$, since (performing again the computation locally) the normal vector is $N(x)=x$, hence its derivative is $dN(v)=v$ and this equals $\widetilde b^{-1}$ applied to $D_v\sigma_x$.
\end{enumerate} 
The computation for $\varsigma$ is completely analogous, implying (under the same identification as above):
\begin{enumerate}
\item[(1')] Its differential equals $(1/2)(\mathbbm 1+\widetilde b)$ and is nonsingular.
\item[(2')] Its Gauss map is again $\dev$.
\item[(3')] Its first fundamental form is $(1/4)\widetilde g_1((\mathbbm 1+ \widetilde b)\cdot,(\mathbbm 1+ \widetilde b)\cdot)=\widetilde G$.
\end{enumerate} 

\item The immersions $\sigma$ and $ \varsigma$ are equivariant with respect to a representation $\rho$ of $\pi_1(U_1^*)\cong\mathbb Z$ into the isometry group of $\E^3$. Indeed, by construction $\widetilde g_1$ and $\widetilde b$ are preserved by the action of $\mathbb Z$ by deck transformations of $\widetilde{U_1^*}$, hence so are the first fundamental form and the shape operator of $ \sigma$, by the items (3) and (4) of the list above. By the uniqueness part of the fundamental theorem of surfaces, there exists a representation 
$$\rho:\mathbb Z\to\Isom(\R^3)\cong\SO(3)\ltimes\R^3$$ such that $\sigma\circ \gamma=\rho(\gamma)\circ\sigma$ for every $\gamma\in\mathbb Z$. 

Now, observe that the developing map $\dev:\widetilde{U_1^*}\to\mathbb S^2$ is also equivariant with respect to a rotation of angle $\theta_1$ (modulo $2\pi$), as a consequence of the definition of spherical cone metric. That is, if we denote by $\gamma_1$ the standard generator of $\pi_1(U_1^*)$, then 
$$\dev(\gamma_1\cdot \widetilde q)=R_{\theta_1}(\dev(\widetilde q))$$
for every $\widetilde q\in \widetilde{U_1^*}$.
Up to composing $\sigma$ and $\varsigma$ with an element of $\SO(3)$, we can assume that $R_{\theta_1}$ is the rotation fixing $(0,0,1)$. 
By item (2) above, the Gauss map of $\sigma$ coincides with $\dev$, hence the linear part of $\rho$ is the holonomy of $\dev$. Concretely, we have for every $\widetilde q\in \widetilde{U_1^*}$.
$$\sigma(\gamma_1\cdot \widetilde q)=R_{\theta_1}(\sigma(\widetilde q))+\tau$$
for some $\tau\in \E^3$. Since $\varsigma=(1/2)(\sigma+\dev)$, it turns out that $\varsigma$ satisfies the equivariance:
$$\varsigma(\gamma_1\cdot \widetilde q)=R_{\theta_1}(\sigma(\widetilde q))+\frac{1}{2}\tau~.$$
As a consequence of the next step, we will see that, up to composing $\sigma$ with a translation, we can assume $\tau=0$.

\item In this step we will show that, roughly speaking, the immersions $\sigma$ and $\varsigma$ admit a limit in correspondence of the puncture of $U_1^*$. Let us denote by $\mathcal D$ a fundamental domain for the action of $\pi_1(U_1^*)$ on $\widetilde{U_1^*}$.  We now claim that there exists a point $\xi\in\E^3$ having the property that $\sigma(\widetilde q_n)\to \xi$ for every sequence $\widetilde q_n\in\mathcal D$ such that $\Pi(\widetilde q_n)\to p_1$, where $\Pi:\widetilde{U_1^*}\to U_1^*$ is the covering projection.  
To see this, recall that the first fundamental form of $\sigma$ is  the lift to the universal cover of the spherical cone metric $\varphi^*g_2$. Let us first fix \emph{one} sequence $\widetilde q_n$ as above. Since the metric completion of $(U_1^*,\varphi^*g_2)$ is obtained by adding the cone point $p_1$, $\widetilde q_n$ is a Cauchy sequence for the first fundamental form of $\sigma$, which is $\widetilde{\varphi^*g_2}$. Hence $\sigma(\widetilde q_n)$ is a Cauchy sequence in $\R^3$, and it converges. Let us call its limit point $\xi$. 

Now pick any other sequence $\widetilde q_n'$ contained in $\mathcal D$ such that $\Pi(\widetilde q_n')$ converges to $p_1$. The distance between $\widetilde q_n$ and $\widetilde q_n'$ for the first fundamental form of $\sigma$ tends to zero. Hence also the Euclidean distance $\| \sigma(\widetilde q_n)- \sigma(\widetilde q'_n)\|$ tends to zero, and therefore the limit of $ \sigma(\widetilde q'_n)$ is $\xi$ again.  One can in fact repeat the same argument only assuming that $\widetilde q_n'$ is contained in the union 
$$\bigcup_{i\in I} \gamma_i\cdot\mathcal D$$
for $I$ a finite subset of $\pi_1(U_1^*)\cong\mathbb Z$. This observation also shows that the representation $\rho$ introduced in the previous step fixes $\xi$, by applying the above argument to $\widetilde q_n'=\rho(\gamma)\widetilde q_n$. Up to composing with a translation, we will assume $\xi=0$, which shows that $\rho$ is a linear representation,  or in other words, $\tau=0$ in the previous step.

We also obtain an analogous property for $\varsigma$, namely that for any sequence $\widetilde q_n\in\mathcal D$ such that $\Pi(\widetilde q_n)\to p_1$, $\varsigma (\widetilde q_n)$ converges to a point in the axis fixed by $R_{\theta_1}$.  But in this case the proof does \emph{not} follow from the same argument. Indeed we do not know that the metric $G$ on $U_1^*$ has a cone singularity at $p_1$, hence we cannot repeat the above argument wordly. (This is indeed the reason why so far we dealt with $\sigma$ and $\varsigma$ simultaneously, although we are only interested in the final statement for $\varsigma$.) 
Nevertheless, since $\varsigma=(1/2)(\sigma+\dev)$ and $\dev$ converges to $(0,0,1)$ on any sequence $\widetilde q_n$ as above, the conclusion for $\varsigma$ follows immediately, and the limit of $\varsigma$ is in fact $(0,0,1)$ under our assumptions.

\item Let us now consider the vertical projection $\pi:\R^3\to\R^2$, namely $\pi(x,y,z)=(x,y)$. 
A fundamental consequence of item (2') is that the Gauss map of  $\varsigma$ (which coincides with $\dev$) tends to $(0,0,1)$ along any sequence $\widetilde q_n$ such that $\Pi(\widetilde q_n)\to p_1$. It follows that, up to restricting $U_1$, we can assume that 
\begin{equation}\label{eq:assumption dev}
\langle\dev(q),(0,0,1)\rangle>\epsilon> 0~.
\end{equation} 
In other words, the normal vector of $\varsigma$ is never horizontal. 
This implies that $\pi\circ\varsigma:\widetilde{U_1^*}\to\R^2$ is an immersion. Since $\varsigma$ is equivariant with respect to the representation $\rho$ sending a generator to the rotation $R_{\theta_1}$ around the vertical axis, also $\pi\circ\varsigma$ is equivariant with a representation, which with a little abuse of notation we still denote by $\rho$, sending the same generator to the rotation in $\R^2$ of angle $\theta_1$ (modulo $2\pi$). By this equivariance, the first fundamental forms of the immersions $\varsigma$ and $\pi\circ\varsigma$ induce two Riemannian metrics on $U_1^*$: the metric induced by $\varsigma$ is $G$ as a consequence of item (3') in Step 1; the metric induced by $\pi\circ\varsigma$, which we call $H$, is flat. The two metrics $G$ and $H$ are bi-Lipschitz, i.e. there exists a constant $C$ such that
\begin{equation}\label{eq:assumption bilip}
\frac{1}{C}H(v,w)\leq G(v,w)\leq CH(v,w)~,
\end{equation} 
for all $v,w$ tangent to $U_1^*$. Indeed, it follows from \eqref{eq:assumption dev} that that is there exists a constant $C>1$ such that
$({1}/{C}){G(v,v)}\leq ||d\pi\circ d\varsigma(v)||^2\leq {G(v,v)}$ 
for any vector $v$ tangent to $\widetilde{U_1^*}$. 
 \item We now claim that the conformal structure on $U_1^*$ induced by the metric $H$ is biholomorphic to $\mathbb D^*$ around the point $p_1$. Recall that any conformal structure on $U_1^*$ is biholomorphic to $\mathbb D^*$, $\C^*$ or $\mathbb A_{r_0}=\{z\in\C\,|\,1<|z|<r_0\}$ for some $r_0>1$. Let us show that $(U_1^*,H)$ cannot be biholomorphic to any $\mathbb A_{r_0}$. Up to restricting the neighbourhood $U_1^*$, we then rule out the case $\C^*$ and conclude the claim.
 
 Suppose by contradiction that there exists a biholomorphism $\psi:\mathbb A_{r_0}\to (U_1^*,H)$. Clearly $\psi$ extends to one of the two boundary components, having limit $p_1$ therein; it is harmless to assume that such boundary component  is $\{|z|=1\}$. Considering the (holomorphic) universal covering map $z\mapsto \exp({-iz})$ of $\mathbb A_{r_0}$, defined on
 \begin{equation}\label{eq:univ cover annulus}
 \widetilde{\mathbb A_{r_0}}:=\{z\in\C\,|\,0<\mathrm{Im}(z)<\log(r_0)\}~,
 \end{equation}
  we can lift $\psi$ to a biholomorphism $\widetilde\psi:\widetilde{\mathbb A_{r_0}}\to (\widetilde{U_1^*},\widetilde H)$, where by construction $\widetilde H$ is the pull-back metric of the immersion  $\pi\circ\varsigma:\widetilde{U_1^*}\to\R^2$. Hence $f:=\pi\circ\varsigma\circ\widetilde\psi:\widetilde{\mathbb A_{r_0}}\to\R^2\cong \C$ is a holomorphic map. In Steps 2 and 3 we showed that $\varsigma(\widetilde q_n)$ tends to  the origin for any sequence $\widetilde q_n\in \widetilde{U_1^*}$ such that $\Pi(\widetilde q_n)\to p_1$. Together with our assumption on $\psi$, it follows that $f$ extends continuously to the real line $\{\mathrm{Im}(z)=0\}$, which is mapped to $0\in\C$. 
 
 Now pick any point $z_0$ on the real line, and pick any $0<\epsilon_0<\log(r_0)$. By Schwarz Reflection Principle, we can extend $f$ on $\widetilde{\mathbb A_{r_0}}\cap B(z_0,\epsilon_0)$ to a holomorphic map $F:B(z_0,\epsilon_0)\to\C$, which is defined on $\{\mathrm{Im}(z)<0\}\cap B(z_0,\epsilon_0)$ by $F(z)=\overline{f(\overline z)}$. Since $f$ maps the real line to $0\in \C$, the holomorphic extension $F$ is constant. But this contradicts the fact that $\pi\circ\varsigma$ is an immersion. 
  
 \item Finally, we can conclude that $G$ has also the conformal type of  $\mathbb D^*$. Indeed, the identity $\mathrm{id}:(U_1^*,G)\to (U_1^*,H)$ is bi-Lipschitz by \eqref{eq:assumption bilip}, hence it is quasiconformal. Since $H$ is biholomorphic to  $\mathbb D^*$ by the previous step, this implies that $(U_1^*,G)$ cannot have the conformal type of $\mathbb A_{r_0}$ or $\C^*$ and is therefore conformal to $\mathbb D^*$, see for instance \cite[\S 4.3]{zbMATH05042912}
\qedhere 
\end{steps}
\end{proof}

\subsection{Proof of Corollary \ref{cor:rigidity}}

Having proved Proposition \ref{prop:conf type}, our proof of Theorem \ref{thm:main1} is complete. We now provide the proof of Corollary \ref{cor:rigidity}.

\begin{proof}[Proof of Corollary \ref{cor:rigidity}]
Suppose $\iota:\Sigma_1\to\E^3$ is a branched immersion, where $\mathfrak p_1$ is the (discrete) branching set, and let $\nu:\Sigma_1\setminus\mathfrak p_1\to\mathbb S^2$ be its Gauss map. It follows from the definition of branched immersion and the compactness of $\Sigma_1$ that $\nu$ extends to a branched covering $\Sigma_1\to\mathbb S^2$. Up to {postcomposition with a homothety of $\E^3$}, we can assume without loss of generality that the value of the constant curvature of $\iota$ is equal to one; we will show that $\iota$ is a branched covering onto a round sphere of radius one centered at some point of $\R^3$.

Let $g_1$ be the first fundamental form of $\iota$ and $b$ its shape operator. Since $b=-d\nu$, we see immediately that $\nu^*g_{\mathbb S^2}=g_1(b\cdot,b\cdot)$, where of course $g_{\mathbb S^2}$ denotes the spherical metric. The $(1,1)$ tensor $b$ is $g_1$-self-adjoint and positive definite (up to changing the sign of $\nu$). Moreover, by the Gauss-Codazzi equations, $b$ satisfies $\det b=1$ and $d^{\nabla^{g_1}}\!b=0$. This shows that $\mathrm{id}:(\Sigma_1,\mathfrak p_1,g_1)\to (\Sigma_1,\mathfrak p_1,\nu^*g_{\mathbb S^2})$ is a minimal Lagrangian diffeomorphism.

We claim that both $g_1$ and $\nu^*g_{\mathbb S^2}$, which are clearly spherical metrics on $\Sigma_1\setminus \mathfrak p_1$, have cone singularities at the points of $\mathfrak p_1$. For $\nu^*g_{\mathbb S^2}$, this is clear since $\nu$ is a local isometry for $\nu^*g_{\mathbb S^2}$ and behaves, in a neighbourhood of any $p\in\mathfrak p_1$, like a degree $d$ covering onto a punctured disc in   $\mathbb S^2$. Hence $\nu^*g_{\mathbb S^2}$ has a cone point of cone angle $2d\pi$. 

For $g_1$, pick any point $p\in\mathfrak p_1$. We know from the definition that both $\iota$ and $\nu$ admit a limit point at $p$. We can assume that the limit of $\iota$ is the origin of $\R^3$, and the limit of $\nu$ is $(0,0,1)$. We can now repeat Steps 4, 5 and 6 of the proof of Proposition \ref{lemma realize min lag} to show that $g_1$ has the conformal type of a punctured disc around $p$. 
Let us briefly summarize these steps for easiness of the reader.  
As in Step 4, for $q$ in a small neighbourhood $U$ of $p$, we have $\langle\nu(q),(0,0,1)\rangle>\epsilon>0$. Hence $\pi\circ \iota$ is a locally bi-Lipschitz immersion with respect to the first fundamental form $g_1$ of $\iota$ and the flat metric of $\R^2$. More precisely $\mathrm{id}:(U,g_1)\to (U,h_1)$, where $h_1$ is the pull-back metric of $\pi\circ \iota$, is bi-Lipschitz. Repeating Step 5, we show that $(U,h_1)$ is biholomorphic to $\mathbb D^*$ around $p$ (up to restricting $U$ to rule out the case $\C^*$): indeed, by lifting a biholomorphism $\psi:{\mathbb A_{r_0}}\to(U,h_1)$ to the universal cover \eqref{eq:univ cover annulus} and applying Schwarz Reflection Principle to $\pi\circ\iota\circ\widetilde\psi$, one would obtain a contradiction. As in Step 6, one then shows that $g_1$ has the conformal type of $\mathbb D^*$ as well. 

Moreover $g_1$ has finite area around $p$. Indeed $h_1$ has finite area: to see this, let $\psi:\mathbb D^*\to (U^*,h_1)$ be a biholomorphism, and consider the holomorphic map $\pi\circ \iota\circ\psi:\mathbb D^*\to\R^2$, that extends to a holomorphic map on $\mathbb D$, hence has finite degree if restricted to a compact neighbourhood of $0$. Hence the area of $h_1$ is finite around $p$, and since $g_1$ is bi-Lipschitz to $h_1$, it has finite area as well. Then we can apply \cite[Proposition 4]{zbMATH04034587} and deduce that $g_1$ has a cone point at $p$, because it is a spherical metric, it is biholomorphic to $\mathbb D^*$ in a neighbourhood of $p$  and has finite area.

Having showed that $g_1$ and $\nu^*g_{\mathbb S^2}$ are spherical cone metrics, it follows from Theorem \ref{thm:main1} that $\mathrm{id}:(\Sigma_1,\mathfrak p_1,g_1)\to (\Sigma_1,\mathfrak p_1,\nu^*g_{\mathbb S^2})$ is an isometry, i.e. $g_1=\nu^*g_{\mathbb S^2}$. Hence $b$ is the identity operator. This means that, considering $\nu$ as an immersion in $\R^3$, its first fundamental form is $g_1$. Moreover its shape operator is the identity, since its image is a subset of $\mathbb S^2$. Hence $\iota$ and $\nu$ are immersions of $\Sigma_1\setminus\mathfrak p_1$ having the same first fundamental form (namely $g_1$) and the same shape operator (namely the identity). By the uniqueness part of the fundamental theorem of surfaces, there exists an isometry $A$ of $\R^3$ such that $\iota=A\circ\nu$. Since $\nu$ is a branched covering of $\mathbb S^2$, this implies that $\iota$ is a branched covering of some round sphere. 
\end{proof}

We observe that in the proof of Corollary \ref{cor:rigidity} we are making use of Theorem \ref{thm:main1} in full strength, namely under the sole assumption that the minimal Lagrangian diffeomorphism $\varphi$ extends \emph{continuously} at the cone points. Indeed, by the definition of branch point for spherical surfaces, the Gauss  map $\nu$ of the immersion $\iota$ is supposed to extend {continuously} at the branch points, without any additional regularity assumption. 

\subsection{Final remarks}
In conclusion, we would like to add some related remarks. 

\begin{remark}
Let $(G,B)$ be the pair constructed in Definition \ref{defi GB}. Recalling  that the (1,1) tensor $B$ is bounded (Lemma \ref{prop eigenvalues}), and $G(B\cdot,\cdot)$ is the real part of a holomorphic quadratic differential $q$ for the conformal structure induced by $G$ (see Proposition \ref{prop hopf}), a simple computation shows that if $\alpha\in (0,1/2]$ (i.e. cone angle in $(0,\pi]$), then $q$ has at most simple poles at $p$. If the cone angle is in $(\pi,2\pi]$ then $q$ does not have a pole, and if the cone angle is in $(k\pi,(k+1)\pi])$ for $k\geq 2$ then $q$ has a zero of order at least $k-1$.
\end{remark}

\begin{remark}\label{rmk same cone angles}
With some more technicalities, similar to those pursued in \cite[\S 4]{zbMATH06561539} in a similar Lorentzian setting, one can strengthen the arguments in the proof of Proposition \ref{lemma realize min lag} and show that the metric $G$ has a cone point of angle $\theta_1$ at $p$. Moreover the immersion $\varsigma$ induces an embedding of a punctured disc in a \emph{singular Euclidean space}, namely the singular Riemannian manifold
$$|z|^{2\alpha_1-2}|dz|^2+dt^2~,$$
which is the product of the standard flat metric on $\R^2$ with cone angle $\theta_1=2\pi\alpha_1$ and $\R$.  This induced embedding is orthogonal to the singular locus $\{z=0\}$ and its first fundamental form of is precisely $G$. 
We stress that this follows from a local analysis, that is under the sole assumption that $\varphi:(U_1,p_1,g_1)\to (U_2,p_2,g_2)$ is a minimal Lagrangian diffeomorphism between discs endowed with spherical cone metrics.
\end{remark}

\begin{remark}\label{rmk extend}
A consequence of Remark \ref{rmk same cone angles} is that, if $\varphi:(U_1,p_1,g_1)\to (U_2,p_2,g_2)$ is a minimal Lagrangian diffeomorphism between discs endowed with spherical cone metrics, then the cone angles of $g_1$ and $g_2$ at $p_1$ and $p_2$ are necessarily equal. Indeed, Remark \ref{rmk same cone angles} shows that the metric $G$ has a cone point of angle equal $\theta_1$ at a point $p_1$; applying the same remark  to $\varphi^{-1}$ (Remark \ref{rmk inverse 2}), one shows that this cone angle equals also $\theta_2$.


This statement holds as well when the ``cone angle'' equals $2\pi$, in the following sense. Suppose $\varphi$ is a minimal Lagrangian diffeomorphism between $(U_1\setminus\{p_1\},g_1)$ and $(U_2\setminus\{p_2\},g_2)$ where  \emph{at least one} between $g_1$ and $g_2$ extends to a nonsingular spherical metric on the whole $U_i$. Then \emph{both} $g_1$ \emph{and} $g_2$ are nonsingular spherical metrics on $U_1$ and $U_2$. 

In this setting, the arguments of the proof of Proposition \ref{lemma realize min lag} (up to switching $\varphi$ and $\varphi^{-1}$, which does not affect the conclusion) can be adapted to show that $\varphi$ can be realized  as the Gauss map of an embedded surface of constant Gaussian curvature one in $\R^3$, and moreover the embedding extends continuously to $p_1$. Classical regularity for Monge-Amp\`ere equations then implies (again only by a local analysis) that  the embedding extend smoothly at $p_1$. Therefore the minimal Lagrangian map $\varphi$ extends smoothly to a minimal Lagrangian diffeomorphism between $U_1$ and $U_2$. 
\end{remark}

\begin{remark}
As mentioned in the end of the introduction, our Theorem \ref{thm:main1} and Corollary \ref{cor:rigidity} (unlike Remarks \ref{rmk same cone angles} and \ref{rmk extend}) cannot be improved to purely local statements, and necessarily require some topological assumption, for instance closedness of the surface. Indeed, examples of non-isometric minimal Lagrangian diffeomorphisms between domains of $\mathbb S^2$ can be found as the Gauss maps of surfaces of constant Gaussian curvature in $\R^3$, for instance as surfaces of revolution. There are also many examples of branched immersions of (necessarily non-closed) surfaces with constant Gaussian curvature one, whose image is not contained in a round sphere (see for instance \cite{zbMATH06599450}). Following the proof of Corollary \ref{cor:rigidity}, their Gauss maps induce minimal Lagrangian diffeomorphisms between open spherical surfaces with cone angles $2n\pi$.
\end{remark}

\appendix

\section{Equivalent definitions}\label{app:defis}

In this appendix we show that Definition \ref{defi min lag} is equivalent to the condition that the graph of $\varphi$ is a minimal Lagrangian submanifold in the product $\Sigma_1\times\Sigma_2$, endowed with the Riemannian metric $g_1\oplus g_2$ and with the symplectic form $\pi_1^*dA_{g_1}-\pi_2^*dA_{g_2}$. This is in fact a local computation.

\begin{prop}\label{prop:app1}
Let $(U_1,g_1)$ and $(U_2,g_2)$ be spherical surfaces, $\varphi:U_1\to U_2$ a diffeomorphism, and $b$ the unique positive definite, $g_1$-self-adjoint $(1,1)$ tensor on $U_1$ such that $\varphi^*g_2=g_1(b\cdot,b\cdot)$. Then the graph of $\varphi$ is minimal Lagrangian if and only if $d^{\nabla^{g_1}}\!b=0$ and $\det b=1$. 
\end{prop}

The proof uses the following well-known characterization of Codazzi tensors, which is also applied in Appendix \ref{app:alternative}:

\begin{prop}[{\cite{zbMATH03064601}}] \label{prop hopf}
Given a Riemannian metric $G$ on a surface $\Sigma$ and a smooth $(1,1)$ tensor $B$, $B$ is $G$-self-adjoint and traceless if and only if $G(B\cdot,\cdot)$ is the real part of a quadratic differential $q$. Moreover $B$ is $G$-Codazzi if and only if $q$ is holomorphic.
\end{prop}

\begin{proof}[Proof of Proposition \ref{prop:app1}]
First, let us observe that the area forms $dA_{g_1}$ and $\varphi^*dA_{g_2}$ differ by the factor $\det b$, hence the graph of $\varphi$ is Lagrangian (i.e. $\varphi$ is area-preserving) if and only if $\det b=1$. In the rest of the proof, we will always assume that both these conditions hold. 

It is known (\cite[\S 4]{zbMATH03198580}) that an immersion $F:U\hookrightarrow U_1\times U_2$ is minimal if and only if it is conformal and harmonic, that is, if and only if it is harmonic with respect to the conformal class of $F^*(g_1\oplus g_2)$ on $U$. Applying this to the graphical immersion $x\mapsto (x,\varphi(x))$, minimality of the graph of $\varphi$ is equivalent to harmonicity of $\mathrm{id}:(U_1,g_1\oplus \varphi^*g_2)\to (U_1,g_1)$ and of $\varphi:(U_1,g_1\oplus \varphi^*g_2)\to (U_2,g_2)$. 

So let us fix $\varphi$, take $b$ as in the statement, and assume moreover $\det b=1$. As in Definition \ref{defi GB}, we introduce the metric $G=(1/4)g_1((\mathbbm 1+b)\cdot,(\mathbbm 1+b)\cdot)$ and the $(1,1)$ tensor $B=(\mathbbm 1+b)^{-1}(b-\mathbbm 1)$. One sees immediately that $B$ is $G$-self-adjoint and traceless. 

It turns out that the metric $G$ is conformal to $g_1\oplus\varphi^*g_2$. Indeed  $b^2-\mathrm{tr}(b)b+(\det b)\mathbbm 1=0$ by the Cayley-Hamilton theorem, hence $\mathbbm 1+b^2=\mathrm{tr}(b)b$, which implies $g_1+\varphi^*g_2=\mathrm{tr}(b)g_1(b\cdot,\cdot)$ and
$$G=\frac{1}{4}\left(1+\frac{2}{\mathrm{tr}(b)}\right)(g_1+\varphi^*g_2)~.$$
Since harmonicity only depends on the conformal class on the source, we will conclude the proof by showing that $d^{\nabla^{g_1}}\!b=0$ if and only if $\mathrm{id}:(U_1,G)\to (U_1,g_1)$ and $\varphi:(U_1,G)\to (U_2,g_2)$ are harmonic.

As a last preliminary step, a direct computation shows 
\begin{equation}\label{eq:g1andG} 
g_1=G((\mathbbm 1-B)\cdot,(\mathbbm 1-B)\cdot)
\end{equation}
for $B=(\mathbbm 1+b)^{-1}(b-\mathbbm 1)$ as in Definition \ref{defi GB}. Using again Cayley-Hamilton, since $B$ is traceless, $B^2=-(\det B)\mathbbm 1$, and it follows that 
$$g_1=\left(1-\det B\right)G-2G(B\cdot,\cdot)~.$$
Since $B$ is $G$-self-adjoint and traceless, by Proposition \ref{prop hopf} $G(B\cdot,\cdot)$  is the real part of a quadratic differential $q$. Then by definition the Hopf differential of $\mathrm{id}:(U_1,G)\to (U_1,g_1)$ is $-q$. Similarly, one writes 
$$\varphi^*g_2=G((\mathbbm 1+B)\cdot,(\mathbbm 1+B)\cdot)=\left(1-\det B\right)G+2G(B\cdot,\cdot)~,$$
hence the Hopf differential of $\varphi:(U_1,G)\to (U_2,g_2)$ equals $q$.

Now, Proposition \ref{prop GCS3} shows that if $d^{\nabla^{g_1}}\!b=0$, then $d^{\nabla^{G}}\!B=0$, which by Proposition \ref{prop hopf} is equivalent to $q$ being holomorphic, and therefore that $\mathrm{id}:(U_1,G)\to (U_1,g_1)$ and $\varphi:(U_1,G)\to (U_2,g_2)$ are harmonic (since $\mathrm{id}$ and $\varphi$ are diffeomorphisms, see \cite[\S 9]{zbMATH03608406}). 

Conversely, suppose that $\mathrm{id}:(U_1,G)\to (U_1,g_1)$ is harmonic, hence $q$ is holomorphic and thus by Proposition \ref{prop hopf} $B$ satisfies the Codazzi condition $d^{\nabla^{G}}\!B=0$. Using \eqref{eq:g1andG} and observing that $b=(\mathbbm 1-B)^{-1}(\mathbbm 1+B)$, one repeats exactly the same proof as in Proposition \ref{prop GCS3} to show that $d^{\nabla^{g_1}}\!b=0$. This concludes the desired equivalence.
\end{proof}

\section{Alternative proof of Lemma \ref{lemma laplace}} 
\label{app:alternative}

In this appendix we provide a new proof of Lemma \ref{lemma laplace}. Denote   $\chi=(1/4)\log(-\det B)$, which is defined in the complement of zeros of $B$. We will show that 
\begin{equation}\label{eq:curvature laplacian}
K_G=\Delta^G\chi~.
\end{equation}

\begin{proof}[New proof of Lemma \ref{lemma laplace}]
Let $q$ be as in Proposition \ref{prop hopf}, and assume that $q$ does not vanish identically. Since \eqref{eq:curvature laplacian} is a local statement, we will work on a local isothermal coordinate $z=x+iy$. Let us first prove the formula for a flat metric $G_0=|dz|^2$ and a $G_0$-Codazzi tensor $B_0$. In this case we have:
$$B_0=G_0^{-1}\mathrm{Re}(q)=\begin{pmatrix}\mathrm{Re}(\phi) & -\mathrm{Im}(\phi) \\  -\mathrm{Im}(\phi)&-\mathrm{Re}(\phi) \end{pmatrix}~.$$
where $q=\phi(z)dz^2$ is the holomorphic quadratic differential as in Proposition \ref{prop hopf}. Hence
$$\chi_0=\frac{1}{4}\log(-\det B_0)=\frac{1}{4}\log(|\phi|^2)~.$$
To show \eqref{eq:curvature laplacian} in this case, it suffices to observe that $\chi_0$ is harmonic with respect to $G_0$. Indeed,
$$\partial_z\chi_0=\frac{1}{4}\partial_z\log(\phi\overline\phi)=\frac{1}{4\phi\overline\phi}\overline\phi\partial_z\phi=\frac{\partial_z\phi}{4\phi}~.$$
Hence $\partial_{\overline z}\partial_z\chi_0=0$ and $\chi_0$ is harmonic.

For the general case, given a Riemannian metric $G=e^{2f}|dz|^2$ and a $G$-Codazzi tensor $B$, as a consequence of Proposition \ref{prop hopf} we have that $B_0:=e^{2f}B$ is Codazzi with respect to the flat metric $G_0=e^{-2f}G=|dz|^2$. Observe that $\det B_0=e^{4f}\det B$, hence
$\chi_0=(1/4)\log(-\det B_0)=\chi+f$, and by the flat case discussed in the first part of the proof, 
$$\Delta^{G_0}\chi_0=\Delta^{G_0}\chi+\Delta^{G_0}f=0~.$$
 Using the formula for the curvature and Laplacian of a conformal metric, we have
$$\Delta^{G}\chi=e^{-2f}\Delta^{G_0}\chi=-e^{-2f}\Delta^{G_0}f=K_G$$
as claimed.
\end{proof}

\bibliographystyle{alpha}
\bibliographystyle{ieeetr}
\bibliography{eesbiblio.bib}

\end{document}